\documentclass[11pt, oneside]{amsart}   	

\usepackage[margin=1.2in]{geometry}

\usepackage{amssymb,mathtools}
\usepackage{amssymb,latexsym,amsmath,amsthm,amscd,mathrsfs,calligra}
\usepackage{enumitem}
\usepackage{extarrows}

\usepackage{tikz-cd}

\usepackage[parfill]{parskip}    		
\usepackage{hyperref}
\hypersetup{
    colorlinks=true,
    linkcolor=blue,
    citecolor=blue
}
\usepackage{cleveref}

\newtheorem{thm}{Theorem}
\newtheorem{prp}{Proposition}

\newtheorem{lem}{Lemma}

\theoremstyle{definition}
\newtheorem{rem}{Remark}

\theoremstyle{definition}
\newtheorem{defn}{Definition}

\newcommand{\SEnd}{\mathscr{E}\kern -.5pt nd}
\newcommand*{\SHom}{\mathscr{H}\kern -.5pt om}
\newcommand{\Rep}{\operatorname{Rep}}
\newcommand{\End}{\operatorname{End}}
\newcommand{\ol}[1]{\overline{#1}}
\newcommand{\hkq}{\mathbin{
  \mathchoice{/\mkern-6mu/\mkern-6mu/}
    {/\mkern-6mu/}
    {/\mkern-5mu/}
    {/\mkern-5mu/}}}
\newcommand{\Met}{\mathrm{Met}}

\newcommand{\cA}{\mathcal{A}}
\newcommand{\cB}{\mathcal{B}}

\newcommand{\cE}{\mathcal{E}}

\newcommand{\cG}{\mathcal{G}}
\newcommand{\cH}{\mathcal{H}}

\newcommand{\cM}{\mathcal{M}}
\newcommand{\cN}{\mathcal{N}}
\newcommand{\cO}{\mathcal{O}}

\newcommand{\cZ}{\mathcal{Z}}

\newcommand{\bbC}{\mathbb{C}}

\newcommand{\bbN}{\mathbb{N}}

\newcommand{\bbP}{\mathbb{P}}

\newcommand{\bbR}{\mathbb{R}}

\newcommand{\bbT}{\mathbb{T}}

\newcommand{\bbZ}{\mathbb{Z}}

\DeclareMathOperator{\tr}{tr}

\DeclareMathOperator{\Hom}{Hom}

\providecommand{\norm}[1]{\ensuremath{\left\lVert#1\right\rVert}}
\providecommand{\abs}[1]{\ensuremath{\left\lvert#1\right\rvert}}

\title{Nakajima Quiver Bundles} 
\author{Lisa Jeffrey \and Matthew Koban \and Steven Rayan}
\date{}			

\newcommand{\Addresses}{{
  \bigskip
  \footnotesize

  {\scshape L.~Jeffrey}, \textsc{Department of Mathematics, University of Toronto,
    Toronto, Ontario, Canada}\par\nopagebreak
     \textit{E-mail address}: \texttt{jeffrey@math.toronto.edu}

  \medskip

  {\scshape M.~Koban}, \textsc{Department of Mathematics, University of Toronto,
    Toronto, Ontario, Canada}\par\nopagebreak
   \textit{E-mail address}: \texttt{matthew.koban@mail.utoronto.ca}

  \medskip

  {\scshape S.~Rayan}, \textsc{Centre for Quantum Topology and its Applications (quanTA) and Department of Mathematics \& Statistics, University of Saskatchewan, 
    Saskatoon, Saskatchewan, Canada}\par\nopagebreak
    \textit{E-mail address}: \texttt{rayan@math.usask.ca}

}}

\begin{document}

\begin{abstract}
    We introduce the notion of a \emph{Nakajima bundle representation}. Given a labelled quiver and a variety or manifold $X$, such a representation involves an assignment of a complex vector bundle on $X$ to each node of the doubled quiver; to the edges, we assign sections of, and connections on, associated twisted bundles. We for the most part restrict attention in our development to algebraic curves or Riemann surfaces. Our construction simultaneously generalizes ordinary Nakajima quiver representations on the one hand and quiver bundles on the other hand. These representations admit gauge-theoretic characterizations, analogous to the ADHM equations in the original work of Nakajima, allowing for the construction of these generalized quiver varieties using a reduction procedure with moment maps. We study the deformation theory of Nakajima bundle representations, prove a Hitchin-Kobayashi correspondence between such representations and stable quiver bundles, examine the natural torus action on the resulting moduli varieties, and comment on scenarios where the variety is hyperk\"ahler. Finally, we produce concrete examples that recover known and new moduli spaces.

\end{abstract}
\maketitle

\setcounter{tocdepth}{1}

\section{Introduction}
 
Motivated by the study of Yang-Mills equations on $4$-dimensional ALE spaces, Nakajima \cite{NakajimaInstantons} had used linear representations of quivers to construct new examples of non-compact hyperk\"ahler varieties in arbitrary dimensions --- the so-called \emph{Nakajima quiver varieties}, which can be thought of as moduli spaces of these representations. This construction builds on a particular narrative threaded through the work of Gibbons-Hawking in type A \cite{GibbonsHawking}; of Atiyah-Drinfel'd-Hitchin-Manin \cite{ADHMinstantons} when the underlying quiver is of ADE type; and of Kronheimer-Nakajima \cite{NK:90} concerning Yang-Mills instantons over asymptotically locally Euclidean (ALE) gravitational instantons. The construction is also closely related to the McKay correspondence for finite subgroups of $\mbox{SU}(2)$ in the case where the underlying quiver is ADE.

Nakajima quiver varieties have been a mainstay in representation theory, in symplectic and complex algebraic geometry, and in mathematical physics. In foundational work \cite{HN:98}, Nakajima connected the aforementioned quiver varieties to integrable highest weight representations of simply-laced Kac-Moody algebras. Specifically, these representations admit an interpretation within the homology ring of these varieties, and the Kac-Moody algebra acts in a certain way via Hecke-type correspondences defined with respect to Lagrangian subvarieties of products of quiver varieties. These subvarieties in turn provide generators for the algebra.  At the same time, variation of stability parameters for the moduli problem, viewed either via GIT or through symplectic reduction, is a means for resolving symplectic singularities \cite{BS:21} --- this includes somewhat surprising cases such as where \emph{every} resolution of a particular quotient singularity in $4$ complex dimensions is realizable as a Nakajima quiver variety for a single quiver \cite{BCRSW:24}. Further still, Nakajima quiver varieties are similar in crucial ways to moduli spaces of stable Higgs bundles: both are non-compact hyperk\"ahler varieties equipped with an algebraic $\bbC^\times$ action, although one is a finite-dimensional quotient (Nakajima) and the other is an infinite-dimensional one (Hitchin). The relation between them is explored from the vantage point of character varieties in \cite{HLRV:11}, as well as in \cite{FisherRayan,RayanSchaposnik} where Nakajima quiver varieties of star- and comet-shaped quivers are embedded into moduli spaces of meromorphic Higgs bundles, thereby eliciting a sub-integrable system of the Hitchin system for certain flag types.

In parallel, Nakajima quiver varieties are crucial objects in symplectic duality. Mathematically, symplectic duality was first observed in \cite{BLPW} as a duality between pairs of symplectic resolutions, of which Nakajima quiver varieties are one of the best understood examples. This duality has been known to physicists since the work of Intriligator \cite{Intriligator} as a duality between $3$-dimensional supersymmetric quantum field theories. Out of such a quantum theory, one can construct two spaces of interest: the Higgs branch and the Coulomb branch. Mathematically, Higgs branches are better understood than the Coulomb branch, and until recently \cite{braverman2016towards} a concrete mathematical definition was not known. When the underlying quantum theory is a quiver theory, the Higgs branch recovers a (singular) Nakajima quiver variety. These developments interact closely with the appearance of Nakajima quiver varieties in symplectic resolutions.

A fundamental question concerns the behaviour of such quiver varieties when the linear representation data is generalized or ``globalized'' to other types of representations, such as those valued in vector bundles over a fixed variety. This is already motivated in the literature: first, in \cite{HitchinSelf-Duality} there is the first study of fixed points of the $S^1$-action on the moduli space of Higgs bundles on a Riemann surface $X$. Each fixed point is a Higgs bundle with a splitting of the underlying bundle for which the Higgs field has a strictly sub-diagonal matrix representation. As such, viewed through the lens of representation theory, the fixed points are representations of A-type quivers in the category of holomorphic bundles over $X$. Here, each vertex is assigned a holomorphic vector bundle $\cE_i$, and each edge is assign a section $\phi_i\in H^0(X,\Hom(\cE_i,\cE_{i+1}))$. Such representations are often referred to as ``holomorphic chains''. More generally (beyond $A$ type) they are referred to as \emph{quiver bundles}. Such objects also appear as solutions to quiver vortex equations \cite{HKcorrespondence}, which themselves capture many well-known correspondences in that the existence of solutions is equivalent to the existence of connections satisfying certain curvature constraints. Some recent works on the physics side involving quiver bundles and emphasizing connections with higher-dimensional branes include \cite{CollinucciSavelli,Tbraneswithin,BranesSymmetry}.

This brings us to the goal of the present paper: to produce a variety parametrizing representations of doubled quivers by complex vector bundles. Such representations will be referred to as ``Nakajima representations''. In doing so, quiver bundles will be constructed on a doubled quiver, and we provide a Hitchin-Kobayashi-type correspondence that identifies points in the moduli space of Nakajima representations with stable quiver bundles. This can be thought of in two ways: providing a relative or global version of Nakajima quiver varieties on the surface $X$, or as the study of quiver bundles in a holomorphic symplectic context. Either of these approaches is a point of view which, to the knowledge of the authors, has not been developed or emphasized in prior literature.  We do note the parallel work of Azam and the third-named author \cite{AR:24} that develops similar ideas related to quiver bundles, but in a categorical context around moduli stacks with homotopy-theoretic applications in mind.

\subsection{Set-up}
Consider the quiver $Q$ with a single vertex and no arrows. Upon fixing a rank $r$ and degree $d$, a quiver bundle is a holomorphic vector bundle of rank $r$ and degree $d$ on $X$. By the Narasimhan-Seshadri theorem \cite{NarasimhanSeshadri}, a stable holomorphic vector bundle on $X$ is equivalent to a complex vector bundle equipped with a hermitian connection $A$ whose curvature satisfies the equation $F_A =0$. One could then hope to recover $\cN_X(r,d)$ the moduli space of stable bundles on $X$ having rank $r$ and degree $d$, from representations of this quiver. With this in mind, the representation space for the quiver with a single vertex and no arrows should be given by the space of connections $\cA(\cE)$ that are compatible with the hermitian metric on $\cE$. A variety parametrizing representations is then precisely the moduli space of flat connections, and it was shown in \cite{HKcorrespondence} that such representations correspond to stable vector bundles. 
\par

Consider now the double quiver $\overline{Q}$ constructed from $Q$ by adding an additional copy of each edge, whose orientation is reversed. Since there are no arrows to double, $\overline{Q}$ coincides with the ordinary quiver $Q$. However it still makes sense to ask what representations of $\overline{Q}$ should be. The guiding principle for our approach is that in the construction of the Nakajima quiver variety, the reverse edges are represented by sections in the cotangent space of the original edge. For the ordinary quiver, a representation is (equivalent to) a stable vector bundle on $X$. Deformations of stable bundles $\cE$ are given by $H^1(X,\SEnd(\cE))$, and so by Serre duality the cotangent directions are $H^0(X,\SEnd(\cE)\otimes K)$ where $K$ is the canonical bundle of $X$. Thus the cotangent space of the space of stable bundles consists of all pairs $(\cE,\phi)$ where $\cE$ is a stable bundle and $\phi$ is a holomorphic $\SEnd(\cE)$-valued $1$-form.
\par
 
It was shown in \cite{SimpsonHiggsBundles} that, for curves $X$ of genus at least $2$, a stable Higgs bundle on $X$ is equivalent to a pair $(A,\phi)$, and the moduli space $\cM^\text{Higgs}_X(r,d)$ of Higgs bundles of rank $r$ and degree $d$ is in fact a hyperk\"ahler variety arising as a reduction of the Hitchin equations

\begin{align}\label{eq:Hitchin}
\mu^\cH_\bbR =\sqrt{-1}F_A &+ [\phi,\phi^*] \\
\mu_\bbC^\cH = &\overline{\partial}_\cA\phi \nonumber.
\end{align}


In order to recover Simpson's construction of $\cM^{\text{Higgs}}_X(r,d)$, a representation of the double quiver $\overline{Q}$ should be a compatible connection $A$ and a section $\phi\in \Omega^{1,0}(X,\SEnd(\cE))$. With this data, the space of quiver representations up to equivalence recovers precisely the moduli space of Higgs bundles on $X$.
\par

Now for a general quiver $Q$ associate a hermitian vector bundle $\cE_i$ with connection $A_i$ to each vertex $i$. If $a:v_i\rightarrow v_j$ is an arrow from vertex $i$ to vertex $j$, associate a bundle morphism $x_a:\cE_i\rightarrow \cE_j$ which can be identified with a global section of the bundle $\SHom(\cE_i,\cE_j)$ on $X$. By taking the $(0,1)$ part of the curvature of the connection, we get a holomorphic structure on $\cE_i$ allowing us to make sense of holomorphic sections. A section $x\in \Omega^0(X,\SHom(\cE_i,\cE_j))$ is holomorphic if it is holomorphic with respect to the operators $\overline{\partial}_{A_i}$ and $\overline{\partial}_{A_j}$. In \cite{HKcorrespondence} it was shown that there is a correspondence between representations satisfying the vortex equations and stable $Q$-sheaves. For the double quiver $\overline{Q}$, to capture the doubled info at the vertices, there must also be a section $\phi \in \Omega^{1,0}(X,\SEnd(E))$. For the reverse edges, one assigns a section $y_a$ dual to $x_a$. Taking all equivalence classes of representations of a given quiver which satisfy a curvature condition, \cref{eq:conjecturedmaps1,eq:conjecturedmaps2}, produces a variety $\cM_{\ol{Q}}^{r,d}(\tau)$ we are calling a \emph{Nakajima bundle variety}. These varieties are, in a sense, a global analogue of Nakajima quiver varieties.

\subsection{Organization}
The paper is organized as follows: \Cref{section:Nakajima Bundle Varieties} introduces representations of doubled quivers in the category of complex vector bundles that satisfy a moment map condition. \Cref{section:Quiver Bundles} provides a brief review of quivers and their representation theory before introducing the notion of $\ol{Q}$-bundles and stability for double quivers.  Then in \Cref{section:StablebundlesNakajima} it is shown that stable $\ol{Q}$-bundles correspond to solutions of the moment map equations, establishing a Hitchin-Kobayashi-type correspondence. \Cref{section: Torus Action} is concerned with a natural torus action on $\cM_{\ol{Q}}^{r,d}(\tau)$. When the underlying quiver consists of just a single vertex this action recovers the familiar torus action on the moduli space of Higgs bundles. Finally, \Cref{section:examples} outlines some preliminary examples of stable $\overline{Q}$ bundles. In particular, we show that new classes of stable objects arise from this framework.

\subsection*{Acknowledgments}
The second-named author gratefully acknowledges support from the Simons Center for Geometry and Physics where important progress was made during the $2$nd Simons Math Summer Workshop (``Moduli'') in July 2024. The second-named author was partially supported by an Ontario Graduate Scholarship and by the NSERC Discovery Grants of the other co-authors.  The second- and third-named authors thank Mahmud Azam, Kuntal Banerjee, Eric Boulter, Robert Cornea, and Evan Sundbo for useful discussions, and all three authors acknowledge Eckhard Meinrenken and Nick Rozenblyum for helpful remarks and questions during the early stages of the work.

\section{Nakajima Bundle Varieties}\label{section:Nakajima Bundle Varieties}

In the present section we define the main object of the current paper, Nakajima bundle representations, which are representations of double quivers by complex vector bundles. There is an action of the group $\cG=\prod_{v\in V}\text{Aut}(\cE_v)$ on the space of all such representations leading to real and complex moment maps. Taking quotients by the group action along level sets of the moment maps produces a variety $\cM_{\overline{Q}}(\tau)$ parametrizing Nakajima bundle representations. These varieties may be thought of as families of Nakajima quiver varieties on the Riemann surface $X$. 

A \emph{quiver} $Q$ consists of a tuple of data $(V(Q),E(Q),h,t)$, where $V(Q)=\{v_0,\cdots, v_n\}$ is the set of vertices, $E(Q)$ the set of edges $a:v_0\rightarrow v_1$, and $h,t:E(Q)\rightarrow V(Q)$ are functions, referred to as the head and tail maps respectively. If $a:v_0\rightarrow v_1$ is an edge starting from vertex $v_0$ and ending on vertex $v_1$, then the maps $t,h$ are defined as  $t(a)=v_0$ and $h(a)=v_1$.

Associated to any quiver $Q$, one can construct the double quiver $\ol{Q}$. The double quiver is a new quiver whose vertices and edges satisfy $V(\ol{Q}) = V(Q)$ and $E(\ol{Q}) = E(Q) \sqcup -E(Q)$, where $-E(Q)$ consists of each element of $E(Q)$ only with opposite orientation.

In order to define representations of a quiver, we require a set of labels. Given a quiver $Q$, a label of $Q$ is a pair $(r,d)$ where $r=(r_v)_{v\in V(Q)} \in \bbN^{\abs{V(Q)}}$ is the rank vector and $d=(d_v)_{v\in V(Q)}\in \bbZ^{\abs{V(Q)}}$ is the degree vector.

\begin{defn}
  Let $\overline{Q}$ be a double quiver. For each $v\in V(Q)$ let $\cE_v$ be a complex vector bundle over $X$. A \emph{Nakajima bundle representation} of $\overline{Q}$ is a collection $(A,\phi,x,y) = (A_v,\phi_v,x_a,y_a)$ indexed by the vertices and edges of $Q$. Here $A_v$ is a hermitian connection on $\cE_v$, $\phi_v\in \Omega^{1,0}(X,\SEnd(\cE_v))$, $x_a\in \Omega^0(X,\Hom(\cE_{t(a)},\cE_{h(a)}))$, and $y_a\in H^1(X, \Hom(\cE_{t(a)},\cE_{h(a)}))$.
\end{defn}


Let $\cA(\cE_v)$ denote the space of unitary connections on $\cE_v$. The space $\text{Rep}(\overline{Q})$ of Nakajima bundle representations of $\ol{Q}$ is the infinite-dimensional affine space 

\[\begin{array}{c}
\prod_{v\in V(Q)} \left(\cA(\cE_v) \times \Omega^{1,0}(X,\SEnd(\cE_v)) \right)\\\bigoplus_{a \in E(Q)}\left( H^0(X,\SHom(\cE_{t(a)},\cE_{h(a)})) \; \oplus \; H^1(X,\SHom(\cE_{h(a)},\cE_{t(a)})\otimes K)\right).\end{array}
\]

At each vertex $v$, the group $\cG_v\coloneqq \text{Aut}(\cE_v)$ of unitary bundle automorphisms acts on $\cE_v$, producing an action of the product group $\cG=\prod_{v\in V(Q)} \cG_v$ on $\text{Rep}(\overline{Q})$:

\[
g\cdot(A,\phi,x,y) = (g\cdot A, g_{v}\phi_v g^{-1}_{v}, g_{h(a)} x_ag^{-1}_{t(a)}, g_{t(a)} y_ag^{-1}_{h(a)}),
\] 

\noindent where $g\cdot A$ is the usual action on connections $g_vd_Ag_v^{-1}$.
\par

In \cite{HitchinSelf-Duality} it was shown that the action of $\cG_v$ on $\cA_v\times \Omega^{1,0}(X,\SEnd(\cE_v)\otimes \bbC)$ is Hamiltonian with real and complex moment maps given by the Hitchin equations \cref{eq:Hitchin}. This extends to the action of the product group $\cG$ on 

\[
\prod_{v\in V(Q)} \cA(\cE_v,h_v) \times \Omega^{1,0}(X,\SEnd(\cE_v)\otimes \bbC)
\]

\noindent whose real and complex moment maps are $\mu_\bbR^\cH =\sum_{v\in V(Q)}\sqrt{-1}F_{A_v}+[\phi_v,\phi_v^*]$ and 
\newline $\mu_\bbC^\cH = \sum_{v\in V(Q)}\overline{\partial}_{\cE_v}\phi_v$.

\subsection{Group action on edge representations}

Before studying the group action on the space of edge representations, we make note of some preliminaries. Let the space of edge representations be denoted $\text{Rep}(E(\overline{Q}))$.

\[
\text{Rep}(E(\overline{Q})) := \bigoplus_{a \in E(Q)} \left( H^0(X,\SHom(\cE_{t(a)},\cE_{h(a)})) \; \oplus \; H^1(X,\SHom(\cE_{h(a)},\cE_{t(a)})\otimes K)\right).
\]

The Hermitian metrics $H_v$ on each bundle $\cE_v$ induce a metric $H$ on $\text{Rep}(E(\overline{Q}))$ by $H(\zeta,\zeta') = \sum_v H_v(\zeta_v,\zeta_v')$. Given a smooth section $x\in \Omega^0(X,\SHom(\cE_i,\cE_j))$, the hermitian metrics on $\cE_i$ and $\cE_j$ define an adjoint $x^*\in \Omega^0(X,\SHom(\cE_j,\cE_i))$. Fix an open subset $U\subseteq X$, there is an $L^2$-inner product on the space of sections $\Gamma(U,\cE_v)$ so that

\[
(\xi_v,\xi'_v)_{L^2,H_v} = \int_U(\xi_v,\xi'_v)_{H_v} =\int _X (\xi_v,\xi'_v)_{H_v},
\]

\noindent with $L^2$-norm $\norm{\xi_v}^2_{L^2,H_v}=(\xi_v,\xi_v)_{L^2,H_v}$. Summing the inner products over all vertices $v$ yields the inner product on $\cE=\bigoplus_{v\in V(Q)}\cE_v$. The $L^p$ norm is defined analogously.  The metrics $H_v$ also induce metrics on the spaces $\Omega^0(X,\SHom(\cE_i,\cE_j))$ by $(\psi,\psi') = \int_X \operatorname{tr}(\psi_a\circ \psi_a^*)$. In particular these constructions will be important in \Cref{section:StablebundlesNakajima}.
\par

There are two symplectic forms $\omega_\bbR$ and $\omega_\bbC$ on $\text{Rep}(E(\overline{Q}))$:

\begin{align*}
\omega_\bbR((x,y),(x',y')) &= \sum_{a\in E(Q)} \int_X (\tr ((x_a(x_a')^* -y_a')^*y_a)\otimes 1_k)) \\
\omega_\bbC((x,y),(x',y')) &= \sum_{a\in E(Q)} \int_X \left( \tr((x_a\otimes \text{id}_K)y'_{a} -(x'_a\otimes \text{id}_K)y_{a}) \right).
\end{align*}

Recall the action of the group $\cG$ of unitary automorphisms of $\text{Rep}(E(\overline{Q}))$, where $g\in \cG$ acts by

\[
g\cdot (x,y) = (g_{h(a)}x_ag_{t(a)}^{-1}, g_{t(a)}y_{a}g_{h(a)}^{-1}).
\]

\begin{prp}The $\cG$-action is Hamiltonian with real and complex moment maps $\mu_\bbR,\mu_\bbC$ whose components at the vertex $i$ are given by

\begin{align}\label{eq:conjecturedmaps1}
\mu(x,y)_{i,\bbR} = \sum_{h(a)=i} \left( x_ax_a^* - y_{a}^*y_{a} \right)\otimes 1_K + \sum_{t(a)=i} \left( x_a^*x_a - y_{a}y_{a}^* \right)\otimes 1_K
\end{align}

\noindent and 

\begin{align}\label{eq:conjecturedmaps2}
\mu(x,y)_{i,\bbC} = \sum_{h(a)=i} (x_a \otimes \operatorname{id}_{K}) y_a -\sum_{t(a)=i} y_a (x_a \otimes \operatorname{id}_{K_a}).
\end{align}

\noindent Where the composite $ss'$ is given by composing the homomorphism and take the wedge product of the differential form. Since we have $y$ in degree $1$ sheaf cohomology, here we are looking at the restrictions of the $x_a$ and $\operatorname{id}_{K}$ to the overlaps $U_\alpha \cap U_\beta$ of coordinate charts on $X$.
\end{prp}

\begin{proof}
Without loss of generality assume that $Q$ consists of a single pair of vertices with one edge between them. Then $\text{Rep}(E(\overline{Q}))= T^*\Omega^0(X,\Hom(\cE_v,\cE_u))$ and the group $\cG$ is $\cG(\cE_v)\times \cG(\cE_u)$. First it will be shown that the complex moment map $\mu_\bbC$ is given by \cref{eq:conjecturedmaps2}. Given an element $\theta\in \cG$ and a representation $x,y$, the component $\theta^x$ of the resulting vector field is given by 

\begin{align*}
\theta^x & = \frac{d}{dt}\lvert_{t=0} \text{exp}(t\theta)\cdot x \\
	     & = \frac{d}{dt}\lvert_{t=0} \text{exp}(t\theta_{h(a)})x \text{exp}(t\theta_{t(a)}) \\
	     &= \theta_{h(a)}x - x\theta_{t(a)}.
\end{align*}

Denote this vector field by $\theta^x =[\theta,x]$. Similarly we have $\theta^y=[\theta,y]\coloneqq\theta_{t(a)}y-y\theta_{h(a)}$. Let $\theta_v$ denote the component of $\mathfrak{g}$ at the vertex $v$. Then $\theta_v$ are skew-adjoint. To find the moment map, we must provide a function $f:\text{Rep}(E(Q))\rightarrow \bbC$ such that $\omega([\theta,x],z) = df_x(z)$. 
Define the function $f$ by $f(x,y) = \sum_i \int_X \operatorname{tr}(\theta_i \mu_\bbC(x,y)_i)$, where $\mu_\bbC(x,y)$ is as in \cref{eq:conjecturedmaps2}


Since trace is linear, and $d\mu_{(x,y)}(z,w) = \sum( xw+zy) + \sum (yz+wx)$. Using this to compute $df_x(z)$: 

\begin{align*}
df_x(z) &= \sum_i \int_X \operatorname{tr}(\theta_i d\mu_x(z)) \\
	   &= \sum_i \int_X \operatorname{tr}(\theta_i (\sum_{h(a)=i} (x_az_{a^*} +z_ax_{a^*}) - \sum_{t(a)=i} (x_{a^*}z_a + z_{a^*}x_a))) \\
	   &= \sum_i \int_X \operatorname{tr}(\sum_{h(a)=i} (\theta_{h(a)}x_az_{a^*} + \theta_{h(a)}z_a x_{a^*}) - \sum_{t(a)=i}( \theta_{t(a)}x_{a^*}z_a + \theta_{t(a)}z_{a^*}x_a)) \\
	   &= \omega_\bbC([\theta,x],z).
\end{align*}

\noindent where the last equality uses both linearity and cyclicity of tr$(A,B)$. Hence the complex moment map has the desired form.
\par

The same computation using the function $f_\bbR(x)=\sum_i\int_X \tr(\theta_i\mu_\bbR(x,y)_i)$, and $\mu_\bbR$ given by \cref{eq:conjecturedmaps1} recovers the desired real moment map.
\end{proof}

Therefore the action of the group $\cG$ on $\text{Rep}(\overline{Q})$ is Hamiltonian, and the moment maps are given by the sum of the moment maps coming from the vertices and that from the edges. The moduli space parametrizing Nakajima representations of $\overline{Q}$ can now be defined.

\begin{defn}
Let $\cZ(\mathfrak{g})$ denote the center of $\mathfrak{g}=\text{Lie}(\cG)$. For a type vector $(r,d)$ and $\tau= (\tau_\bbR,\tau_\bbC) \in \cZ(\mathfrak{g}) \oplus (\cZ(\mathfrak{g})\otimes \bbC)$ the \emph{Nakajima bundle variety at level} $\tau$, $\cM_{\overline{Q}}^{r,d}(\tau)$ is the moduli space of Nakajima representations which are solutions of the equations

\begin{align}
\mu_\bbR(A,\phi,x,y)_i &: \mu^\cH_{\bbR,i} + \sum_{h(v)=i} \left( x_vx_v^* - y_{-v}^*y_{-v} \right)\otimes 1_K + \sum_{t(v)=i} \left( x_v^*x_v - y_{-v}y_{-v}^* \right)\otimes 1_K = \tau  \label{eq:1}\\
\mu_\bbC(A,\phi,x,y)_i	  &: \mu_{\bbC,i}^\cH +  \sum_{h(v)=i} (x_v\otimes 1_K)y_v -\sum_{t(v)=i} y_v (x_v \otimes 1_K) =\tau.\label{eq:2}
\end{align}

That is, it is the quotient

\[
\cM_{\overline{Q}}^{r,d}(\tau) := \Rep(\overline{Q}) \hkq_\tau \cG = \left( \mu_{\bbR}^{-1}(\tau) \cap \mu_\bbC^{-1}(0) \right) / \cG.
\]

\end{defn}

\begin{rem}
  Despite having the appearance of a hyperk\"ahler reduction, the moduli spaces $\cM_{\ol{Q}}^{r,d}(\tau)$ may not have the structure of a hyperk\"ahler variety in general. The reason for this being that $\text{Rep}(\ol{Q})$ is not guaranteed to possess the necessary quaternionic structure. However, in many examples the quotient recovers a hyperk\"ahler variety. It is expected that if the genus of $X$ is at least $2$, and in the absence of any twisting line bundles, $\cM_{\ol{Q}}^{r,d}(\tau)$ will be hyperk\"ahler. 
\end{rem}

To compute the dimension of $\cM_{\ol{Q}}^{r,d}(\tau)$, one must examine the first-order deformations of a tuple $(A,\phi,x,y)$. Such a deformation is a tuple $(\dot{A},\dot{\phi},\dot{x},\dot{y})$ in the direct sum

\begin{equation}\label{eq:deformation}
\begin{array}{c}\bigoplus_{v\in V(Q)} \Omega^1(X,\SEnd(\cE_v)) \oplus \Omega^{1,0}(X,\SEnd(\cE_v))\\\oplus\\\bigoplus_{a\in E(Q)} \Omega^0(X, \SHom(\cE_{h(a)},\cE_{t(a)})) \oplus \Omega^{1,0}(X,\SHom(\cE_{t(a)},\cE_{h(a)})\otimes K).\end{array}
\end{equation}

Determining the space of all deformations involves computing the hypercohomology groups of the deformation complex associated with \cref{eq:deformation}. In this case, the resulting complex contains a grading by $\bbZ^4$. An alternative approach follows \cite[Sec. 5]{HitchinSelf-Duality} where index theory was used to compute the deformation space at a point in the moduli space of stable Higgs bundles. In \Cref{section:examples} the dimension of $\cM_{\ol{Q}}^{r,d}(\tau)$ is computed for specific examples.

We close this section with a few remarks regarding the definition above.
\begin{enumerate}[itemsep=0pt]
    \item The spaces $\cM_{\ol{Q}}^{r,d}(\tau)$ may depend on the level $\tau$. It is expected that $\cZ(\mathfrak{g})$ possesses a decomposition into chambers such that if $\tau$ and $\theta$ lie in the same chamber, the moduli space at level $\tau$ will be isomorphic to that at level $\theta$. The chamber structure of $\cM_{\overline{Q}}(\tau)$ will be left for future work.
    \item The ``Higgs fields'' $\phi_v$ are in general no longer holomorphic forms but holomorphic up to contributions from the $x_a,y_a$. Nakajima representations can therefore be considered to be a generalization of Higgs bundles.
    \item  Upon restricting to a point $x\in X$, and looking at the fibres $\cE_{v,x}$, a $\overline{Q}$-bundle restricts to a linear representation of the double quiver. Thus when the underlying manifold is a point $\cM_{\ol{Q}}^{r,d}(\tau)$ is a Nakajima quiver variety $\cM_{\overline{Q},p}(\tau)$. In this sense, a Nakajima bundle variety is a family of quiver varieties over $X$. However, this does not recover the representation of the quiver one might expect. When $\cE_v$ is a bundle over a genus $g\geq 2$ surface, whose rank is at least $2$ the presence of the Higgs fields $\phi$ serves to add a loop at vertex $v$.
    \item As is the case for usual Nakajima quiver varieties, the inclusion $\mu^{-1}_\bbR(\tau) \cap \mu^{-1}_\bbC(0) \rightarrow \mu^{-1}_\bbC(0)$, followed by the quotient map $\mu^{-1}_\bbC(0)\rightarrow \mu^{-1}_\bbC(0)//G$, gives a proper map $\cM_{\overline{Q}}(\tau) \rightarrow \cM_{\overline{Q}}(0)$.
\end{enumerate}

\section{$\ol{Q}$-Bundles}\label{section:Quiver Bundles}

The equations \cref{eq:1,eq:2} defined in the previous section may be thought of as vortex equations on a doubled quiver. In \cite{HKcorrespondence} it was shown that solutions of the vortex equations correspond to polystable quiver bundles. A quiver bundle over $X$ is a collection $(\cE_v,x_v)$ consisting of a holomorphic vector bundle $\cE_v$ over $X$ for each vertex $v\in V(Q)$, along with a bundle morphism $x_a: \cE_{t(a)}\longrightarrow \cE_{h(a)}$ for each edge $a \in V(Q)$. The main purpose of this section is to define the appropriate notion of a quiver bundle for a double quiver $\ol{Q}$, which we refer to as $\ol{Q}$-bundles, such that polystable $\ol{Q}$-bundles provide solutions to the doubled vortex equations. One of the main differences in the current work lies in the fact that the sections associated with edges in the quiver are no longer taken to be strictly holomorphic.

\begin{defn}
    Given a quiver $Q$ with double $\overline{Q}$, and labels $(r,d)$, a $\overline{Q}$-bundle $(\cE,\phi,x,y)$ over $X$ is a collection $(\cE_v,\phi_v,x_a,y_a)$ indexed by the vertices and edges of $\overline{Q}$, where $\cE_v$ is a holomorphic vector bundle of rank $r_v$ and degree $d_v$ on $X$, $\phi_v\in \Omega^{1,0}(X,\SEnd(\cE_v)\otimes \bbC)$, $x_a\in \Hom(\cE_{t(a)},\cE_{h(a)})$, and $y_a\in \Hom(\cE_{t(a)},\cE_{h(a)})^*$. A $\overline{Q}$-bundle $(\cE,\phi,x,y)$ is often denoted simply by $\cE$ when the context is clear.
\end{defn}

A morphism of $\overline{Q}$-bundles $f:(\cE,\phi,x,y)\rightarrow (\cE',\phi',x',y')$, is a collection $f_v:\cE_v \rightarrow \cE'_v$ of bundle morphisms such that the required squares commute for all $v$ and $a$.

\[\begin{tikzcd}
	{\mathcal{E}_v} & {\mathcal{E}_v} & {} & {\mathcal{E}_{t(a)}} & {\mathcal{E}_{t(a)}} && {\mathcal{E}_{t(a)}} & {\mathcal{E}_{h(a)}} \\
	{\mathcal{E}_v'} & {\mathcal{E}_v'} && {\mathcal{E}_{t(a)}'} & {\mathcal{E}_{t(a)}'} && {\mathcal{E}_{t(a)}'} & {\mathcal{E}_{h(a)}'}
	\arrow["{\phi_v}", from=1-1, to=1-2]
	\arrow["{f_v}"', from=1-1, to=2-1]
	\arrow["{f_v}", from=1-2, to=2-2]
	\arrow["{x_a}", from=1-4, to=1-5]
	\arrow["{f_{t(a)}}"', from=1-4, to=2-4]
	\arrow["{f_{h(a)}}", from=1-5, to=2-5]
	\arrow["{f_{t(a)}}"', from=1-7, to=2-7]
	\arrow["{y_a}"', from=1-8, to=1-7]
	\arrow["{f_{h(a)}}", from=1-8, to=2-8]
	\arrow["{\phi_v'}"', from=2-1, to=2-2]
	\arrow["{x'_a}"', from=2-4, to=2-5]
	\arrow["{y_a'}", from=2-8, to=2-7]
\end{tikzcd}\]

The composition of morphisms $f:(\cE,\phi,x,y) \rightarrow (\cE',\phi',x',y')$ and $g:(\cE',\phi',x',y') \rightarrow (\cE'',\phi'',x'',y'')$ is given by composing the individual bundle morphisms $g_v\circ f_v$. The identity morphism is given by the collection $1=(1_v)_{v\in V(q)}$, where $1_v$ is the identity morphism on $\cE_v$.

\begin{defn}\label{defn:subbundle}
If $(\cE,\phi,x,y)$ is a $\overline{Q}$-bundle, a $\overline{Q}$-subbundle consists of the data $(\cE',\phi',x',y')$ where $\cE'_v\subseteq \cE_v$ for all $v$, and such that the sections $\phi,x,y$ agree with $\phi',x',y'$ when restricted to the $\cE'_v$. An $\ol{Q}$-subbundle will often be referred to as a subbundle when there is no possibility of confusion.
\end{defn}

In order to identify $\overline{Q}$-bundles with solutions to \cref{eq:conjecturedmaps1,eq:conjecturedmaps2}, it is necessary to restrict to certain subsets of \emph{stable} $\overline{Q}$-bundles.

\begin{defn}
    Let $\cE=(\cE,\phi,x,y)$ be a $\overline{Q}$-bundle on a fixed Riemann surface $X$. For tuples of real numbers $\sigma=(\sigma_v)_{v\in V(Q)}$, $\tau= (\tau_v)_{v\in V(Q)}$ with $\sigma_v>0$, define the degree and the $\tau$-rank of $\cE$, $\operatorname{deg}(\cE),\operatorname{rk}_\tau(\cE)$
    
    \[
    \sigma\cdot \operatorname{deg}(\cE) \coloneqq \sum_{v\in V(Q)}\sigma_v \operatorname{deg}(\cE_v),\quad \operatorname{rk}_\tau(\cE) =\tau\cdot \operatorname{rk}(\cE) \coloneqq \sum_{v\in V(Q)} \tau_v \operatorname{rk}(\cE_v).
    \]
    
    The $(\sigma,\tau)$-slope of $\cE$, which depends on $\sigma,\tau$ is defined as

    \begin{equation}
       \mu_{\sigma,\tau}(\cE) = \frac{\sigma\cdot\operatorname{deg}(\cE)+ \tau \operatorname{rk}(\cE)}{\sum_{v\in V(Q)}\operatorname{rk}(\cE_v)}. 
    \end{equation}

A $\overline{Q}$-bundle $\cE$ is \emph{semi-stable} if for all subbundles $(\cE',\phi',x',y')$, the slope satisfies $\mu_{\sigma,\tau}(\cE') \leq \mu_{\sigma,\tau}(\cE)$. It is said to be \emph{stable} if the inequality is strict, and \emph{polystable} if $\cE$ is a direct sum of stable $\overline{Q}$-bundles all of which have the same slope.
\end{defn}

\begin{rem}\label{rem:1}
    It is always possible to set $\mu_{\sigma,\tau}(\cE)=0$ and take $\sigma=1$. Hence proving stability amounts to showing $\mu_{\sigma,\tau}(\cE')<0$ for all subbundles $\cE'$.  
\end{rem}

Stable $\overline{Q}$-bundles behave similarly to stable vector bundles. In particular, a standard proof as in \cite[5.7.12]{Kobayashi87} shows that stable $\overline{Q}$ bundles have no non-scalar automorphisms.

\begin{prp}\label{prp: stable implies simple}
If $(\cE,\phi,x,y)$ is stable, then it is simple.
\end{prp}

When $Q$ is the quiver with one vertex and no edges then a $Q$-bundle is nothing but a vector bundle on $X$, and a $\overline{Q}-$bundle is a Higgs bundle on $X$. For an arbitary quiver, if the sections $y$ associated to the reverse edges vanish, a $\overline{Q}$-bundle is an ordinary quiver bundle. Moreover, one could consider twisting the bundle morphisms $x,y$ by some non-trivial line bundle $L$, hence the theory of $\ol{Q}$-bundles recovers many familiar objects as special cases.

\section{Stable $\overline{Q}$-bundles and Nakajima Representations}\label{section:StablebundlesNakajima}

The main result of the current section is a proof of the correspondence between Nakajima representations and $\overline{Q}$-bundles. This correspondence will identify points in the Nakajima bundle variety $\cM_{\overline{Q}}(\tau)$ with stable $\overline{Q}$-bundles. First it is shown that if a $\overline{Q}$-bundle admits a representation which satisfies \cref{eq:1}, then it must be polystable. The reverse identification is more involved, and will require a good deal of set up. The method of proof in both cases is similar to those used in the proofs of similar Hitchin-Kobayashi type correspondences.

\subsection{Solutions to moment map give stable objects}\label{subsection:moment map implies stable}

If $(\cE,\phi,x,y)$ is a $\overline{Q}$-bundle, a hermitian connection $A$ on $(\cE,\phi,x,y)$ is a collection $(A_v)_{v\in V(Q)}$, where $A_v$ is a hermitian connection on $\cE_v$. The curvature $F_A$ of a connection $A$ on $\cE$ is the collection $(F_{A_v})_{v\in V(Q)}$ of the curvatures of $A_v$.

Before proceeding we require the following description of the degree of a subbundle. Let $(\cE,\phi,x,y)$ be a $\overline{Q}$-bundle equipped with a hermitian metric $H$. Let $(\cE',\phi',x',y')$ be a subbundle. If $\pi:\cE\rightarrow \cE'$ denotes the $H$-orthogonal projection, then using the Chern-Weil formula, \cite[Lemma 3.2]{SimpsonVariationHodge}

\begin{equation}\label{eq:subdegree}
\text{deg}(\cE') = \left( \text{Vol}(X)^{-1} \int_X \operatorname{tr}(\pi \sqrt{-1}F_{A_H}) -\norm{\overline{\partial}_{\cE}(\pi)} \right). 
\end{equation}  

\begin{prp}\label{prp: moment map implies stable}
If $(A,\phi,x,y)$ satisfies the moment map equations, then the corresponding $\overline{Q}$-bundle $(\cE,\phi,x,y)$ is polystable. 
\end{prp}

\begin{proof}
Let $(\cE,\phi,x,y)$ be a $\ol{Q}$-bundle admitting a hermitian metric $\cA$ which satisfies \cref{eq:conjecturedmaps1,eq:conjecturedmaps2}. We may assume without loss of generality that $\cE$ is not a direct sum of subbundles. It must be shown that $\cE$ is stable. By \Cref{rem:1}, it is enough to show that $\mu_{\sigma,\tau}(\cE')<0$ for any proper subbundle $(\cE',\phi',x',y')$. Let $\cE'\subset \cE$ be a proper subbundle with projection $\pi:\cE\rightarrow \cE'$. Fix a vertex $k\in V(Q)$ and consider the sum 

\[
\sum_{a\in E(Q):k=h(a)}\int_X \operatorname{tr}((x_ax_a^*-y^*_ay_a)\pi).
\]

One may think of this sum as the inner product $(xx^*-y^*y,\pi)_{L^2,H}$ restricted to only those edges whose tail is the vertex $k$. By \cref{eq:1},

\begin{align*}
\sum_{a\in E(Q): k=h(a)}&\int_X \operatorname{tr}((x_ax_a^*-y_a^*y_a)\pi) =\sum \int_X \operatorname{tr}(x_a \pi_{t(a)}x_a^*\lvert_{\cE'_k} - \pi_k y_a^*y_a\lvert_{\cE'_k}) \\
									  &=\sum \int_X \operatorname{tr}(\pi_k ( x_a ( \pi_{t(a)} -1) x_a^*\vert_{\cE_k}) - \pi_k\sigma_k\sqrt{-1}\Lambda F_{A_k} - \pi_k[\phi_k,\phi_{k}^*] - \tau_k1_{\cE'_k}) \\
									  &= - \sum \int_X \operatorname{tr}( \pi_k x_a( 1 -\pi_{t(a)})\{ \pi_k x_a (1-\pi_{t(a)})\}^*) - \int_X \operatorname{tr}(\pi_k \sigma_k\sqrt{-1}F_{A_k}) \\
									  & - \sum \int_X \operatorname{tr}(\pi_k[\phi_k,\phi_{k}^*]) - \int_X \operatorname{tr}(\tau_k 1_{\cE'_k}).
\end{align*}

\noindent Now, summing over all $k\in V(Q)$, each term in the left hand side appears twice, with opposite signs. hence 

\begin{eqnarray}\label{eq:8}
~~\norm{\pi_k x_a( 1 -\pi_{t(h)}))}_{L^2,H} + ([\phi,\phi^*],\pi)_H & = & \sum - \int_X \operatorname{tr}(\pi_k \sigma_k\sqrt{-1}F_{A_k}) - \int_X \operatorname{tr}(\tau_k 1_{\cE'_k}) \\ \nonumber
																		     &= & - \sum\text{Vol}(X)\left( \sigma_k\text{deg}(\cE'_k) +\norm{\overline{\partial}_{\cE'}(\pi)}^2\right)\\\nonumber & & -\text{Vol}(X)\tau_k\text{rk}(\cE'_k).
\end{eqnarray}

\noindent One can check that $([\phi,\phi^*],\pi)_H =\norm{\pi\phi(1-\pi)}_{L^2,H}$. By assumption $\cE$ is indecomposable, and so one of $\norm{\pi x(1-\pi)},\norm{\pi\phi(1-\pi)}$ or $\norm{\bar{\partial}_\cE(\pi)}$ must be nonzero. Rearranging \cref{eq:8} then gives

\begin{equation}
0 > \sum_k\left( \sigma_k\text{deg}(\cE_k')- \tau_k\text{rk}(\cE'_k)\right) = \sigma\cdot \text{deg}(\cE) + \tau \text{rk}(\cE).
\end{equation}

\end{proof}

\subsection{Stable implies moment map solution}

The rest of \Cref{section:StablebundlesNakajima} is dedicated to proving the converse to \Cref{prp: moment map implies stable}: that a polystable $\overline{Q}$-bundle $(\cE,\phi,x,y)$ admits a hermitian metric which satisfies \cref{eq:1}. This result is proved by a method of variational calculus, and follows the logic of the proof for quiver bundles \cite{HKcorrespondence} and Higgs bundles \cite{SimpsonVariationHodge}. The main idea of the proof is to introduce a Lagrangian functional $\cM_\tau$, on the space of connections. The functional $\cM_\tau$ is constructed so that the minima of $\cM_\tau$ are precisely solutions to \cref{eq:1}. Following the construction of $\cM_\tau$ the next step is to show the existence of minimizing metrics. Proving existence is accomplished in two steps: First, \Cref{prp:Simplesolves} demonstrates that if a simple $\overline{Q}$-bundle admits a metric minimizing the Lagrangian, it satisfies \cref{eq:1}. Second, if the metric satisfies an $L^1$ bound then it minimizes the Lagrangian. The original work in this section amounts to extending the results in \cite{HKcorrespondence} to the current setting of double quivers in which not all the sections appearing in the equations are globally defined. In situations where results carry over unchanged from previous work, this will be noted. Throughout the remainder of the current section we work with a fixed $\overline{Q}$ bundle $(\cE,\phi,x,y)$ on $X$.

\begin{thm}\label{thm:1}
Suppose $(\cE,\phi,x,y)\in \mu_\bbC^{-1}(\tau_\bbC)$ is a polystable $\overline{Q}$ bundle over a compact Riemann surface $X$. Then $\cE$ admits a hermitian connection satisfying  \cref{eq:1}.
\end{thm}

\subsubsection{Preliminaries}\label{subsection:Preliminaries}

Before proving \Cref{thm:1}, it is necessary to give some preliminaries. Let $\Met_v$ denote the space of hermitian metrics on the vector bundle $\cE_v$. Fix a $K_v\in \Met_v$ such that the induced metric on the determinant line bundle satisfies $\sqrt{-1}F_{\text{det}(K_v)}=\text{deg}(\cE_v)$. Given a hermitian metric $H_v$, $S(H_v)$ denotes the space of smooth $H_v$-selfadjoint endomorphisms of the bundle $\cE_v$, and $S^0(H_v)$ the subspace of trace free endomorphisms. Any metric $H_v\in \Met_v$ is related to $K_v$ through a section $s_v\in S(K_v)$ by $H_v=K_v e^{s_v}$. The adjoint of a morphism with respect to a hermitian metric $H$ is denoted $x^{*_H}$. If the adjoint is with respect to the fixed metric $K_v$, it will be denoted $x^*$. 
\par

Given $x\in H^0(X,\SHom(\cE_v,\cE_{v'}))$ and $x'\in H^0(X,\SHom(\cE_{v'},\cE_v))$, define holomorphic sections $xx'^{*_H}, x^{*_H}x'$ and $[x,x'^{*_H}]$ in $H^0(X,\SEnd(\cE_{v'}))$ as follows:

\[
(xx'^{*_H})_v = \sum_{v\in h^{-1}(a)} x_ax_a'^{*_H}, \quad (x^{*_H}x')_v = \sum_{v\in t^{-1}(a)} x_a^{*H}x_a',\quad [x,x'^{*_H}] = xx'^{*_H} - x^{*_H}x'
\]

\noindent and similarly for the $y$. Also define $xy, yx\in H^1(\SEnd(\cE)\otimes K)$ by 

\[
(xy)_v=\sum_{v\in h^{-1}(a)} (x_a\otimes \text{id}_K)y_a, \;(yx)_v =\sum_{v\in t^{-1}(a)}y_ax_a
\]

where $x_a$ is restricted to two-fold coordinate overlaps. Let $F_H=\sum_{v\in V(Q)} F_{H_v}$, we can rewrite the real moment map \cref{eq:1} as

\begin{equation}
F_H +[\phi,\phi^*] + [x,x^*]\otimes 1_K - [y^*,y] \otimes 1_K = \tau\cdot \text{id}.
\end{equation}

Let $L^p_2S(K_v)$ denote the Sobolev space of $K_v$-selfadjoint sections of class $L^p_2$ and 
\newline 
$\Met(L^p_2S(K_v)) = \{ K_v e^{s_v} : s_v \in L^p_2S_v\}$ be the set of hermitian metrics $H_v$ on $\cE_v$ such that $H_v=K_ve^{s_v}$. This notation extends to the entire $\overline{Q}$ bundle $\cE$ by dropping the reference to the specific vertex $v$.  Given $H =K e^s\in \Met(L^p_2S(K))$, define the $H$-adjoint of a section $\xi$ by $\xi^{*_H}= e^{-s}\xi^* e^s$. Finally, define the connection $A_{H_v}$ by

\begin{equation}\label{eq: connection}
    d_{H_v} = d_{K_v} + e^{-s_v} \partial_{K_v}(e^{s_v})
\end{equation}

\noindent with curvature $F_{H_v} = F_{K_v} + \overline{\partial}_{\cE_v}(e^{-s_v}\partial_{K_v}(e^{s_v}))$.
\par

We will also require the following definitions, due to \cite{HKcorrespondence} and generalizing those of \cite{SimpsonVariationHodge}. Let $\psi:\bbR\rightarrow \bbR$, $\Psi:\bbR\times \bbR\rightarrow \bbR$ be functions and $s\in S$ a self-adjoint endomorphism.  Let $(u_{v,i})_i$ be an orthonormal basis with respect to $K_v$ of the fibre $\cE_{v,x}$ with dual basis $(u^{v,i})^i$. a section $s_v$ can be expressed in terms of the orthonormal bases  $s_v=\sum \lambda_{v,i} u_{v,i}\otimes u^{v,i}$. Define an endomorphism $\varphi(s)\in S$ and linear maps $\Psi(s):S\rightarrow S(\SEnd\cE)$ and $\Phi(s):S\rightarrow S(\mathrm{End}(\cE,\phi,x,y))$ as follows:

\[
\psi(s_v)(p)= \sum_i \psi(\lambda_{v,i})u_{v,i}\otimes u^{v,i}
\]

\noindent and for $f\in S$ with $f_v(p) = \sum_{i,j}f_{v,ij}u_{v,i}\otimes u^{v,j}$, 

\[
\Psi(s_v)f_v(p) = \sum_{i,j} \Psi(\lambda_{v,i},\lambda_{v,j})f_{v,ij}u_{v,i}\otimes u^{v,j}.
\]

\noindent If $\Psi$ is given by $\Psi(p,q)=\varphi_1(p)\varphi_2(q)$ for functions $\varphi_1,\varphi_2:\bbR\rightarrow \bbR$, then

\[
(\Psi(s)\phi)_a = \psi_1(s_{h(a)})\phi_a\psi_2(s_{t(a)})
\]

\noindent and similarly for $\Psi(s)x$ and $\Psi(s)y$.

Lemma $3.1$ of \cite{HKcorrespondence}, repeated below, is an important technical result needed for the proof of \Cref{thm:1}.

We will require the following lemma

\begin{lem}
[\cite{HKcorrespondence}] \hfill
\newline
\begin{enumerate}
    \item For some $b'$, the map $\psi: S \rightarrow S$ extends to a continuous map \newline $\varphi: L^2_{0,b}\rightarrow L^2_{0,b'}$.
    \item For some $b', q \leq 2$, $\varphi$ extends to a map $\varphi: L^2_{1,b}\rightarrow L^q_{1,b'}$. If $q<2$ then it is continuous.
    \item For $q\leq 2$, $\Psi: S\rightarrow S(\SEnd(\cE))$ extends to a map$$\Psi: L^2_{0,b}\rightarrow \Hom(L^2\Omega^0(\SEnd(\cE)),L^q\Omega^0(\SEnd(\cE))),$$which is continuous for $q<2$.
    \item For some $b'$, $\Psi:S\rightarrow S(\End(\cE,\phi,x,y))$ extends to a continuous map $\Psi:L^2_{0,b}S\rightarrow L^2_{0,b'}S(\End(\cE,\phi,x,y))$.
    \item The above maps extend to smooth maps $\psi:L^p_2S\rightarrow L^p_2S,\Psi: L^p_2S\rightarrow L^p_2S(\SEnd(\cE))$, and $\Psi: L^p_2S\rightarrow L^p_2S(\End(\cE,\phi,x,y))$.
\end{enumerate}    
\end{lem}

\subsubsection{Donaldson Lagrangian}\label{subsubsection:Lagrangian}
Now to define a Lagrangian functional on $\Met(L^p_2S(K))$, the minima of which will correspond to solutions of \cref{eq:1}. Our definition builds on Simpson's definition of the Donaldson Lagrangian in \cite[p.883]{SimpsonVariationHodge} and was first used in \cite{DonaldsonStable,DonaldsonYM}. Let $\Psi:\bbR\times \bbR \rightarrow [0,\infty)$ be given by

\[
\Psi(x,y) = \frac{e^{y-x}-(y-x)-1}{(y-x)^2}.
\]

The Donaldson Lagrangian $M_{D,v}=M_D(K_v,\cdot):\Met(L^p_2S(K_v))\rightarrow \bbR$ is 

\begin{equation}
M_{D,v}(H_v) = (\sqrt{-1} F_{K_v}, s_v)_{L^2} + \left( \Psi(s_v)(\overline{\partial}_{\cE_v}s_v),\overline{\partial}_{\cE_v}s_v\right)_{L^2}.
\end{equation}

The Donaldson Lagrangian can be extended to a functional on $\Met(L^p_2S(K))$ as 
\newline $M_D(H) = \sum_{v\in V(Q)} M_{D,v}(H_v)$. 

\begin{defn}\label{defn:Lagrangian}

Let $\cM_{\tau}(\cdot):\Met(L^p_2S(K_v))\rightarrow \bbR$ be defined by the formula

\begin{equation}
\cM_{\tau}(H) = M_D(H) +\norm{\phi}^2_{L^2,H} -\norm{\phi}^2_{L^2,K} +\norm{x}^2_{L^2,H} -\norm{x}^2_{L^2,K} + \norm{y}^2_{L^2,H} -\norm{y}^2_{L^2,K} - (s,\tau\cdot \text{id})_{L^2}.
\end{equation}
\end{defn}

\begin{lem}
If $H\in \Met(L^p_2S)$, and $\psi:\bbR\times \bbR\rightarrow \bbR$ is given by $\psi(p,q)=e^{p-q}$ then

\[
\cM_{\tau}(H) = M_D(H) +(\psi(s)\phi,\phi)_{L^2} -\norm{\phi}_{L^2}^2 + (\psi(s)x,x)_{L^2} - \norm{x}^2_{L^2} + (\psi(s)y,y)_{L^2} - \norm{y}^2_{L^2}.
\]
\end{lem}

\begin{proof}
The only term which does not follow immediately from \cite[lem 3.2]{HKcorrespondence} is $(\psi(s)y,y)_{L^2}$. By definition
$y_a^{*H_a}= (e^{-s_{h(a)}})\circ y^{*K_a}\circ e^{s_{t(a)}}$ and $(\psi(s)y)_a = e^{s_{t(a)}}\circ y_a \circ e^{-s_{h(a)}}$. Then we have

\begin{align*}
\norm{y_a}^2_{L^2,H_a} = \int_X \operatorname{tr}(y_ay_a^{*H_a}) &= \int_x \operatorname{tr}(y_a \circ (e^{-s_{ha}}\circ y_a^{*K_a}\circ e^{s_{ta}})) \\
&= \int_X \operatorname{tr}(e^{s_{ta}}\circ y_a \circ e^{s_{ha}}y_a^{*K_a}) \\
&= ((\psi(s)y)_a,y_a)_{K_a}.
\end{align*}

\end{proof}

Now, define $m_{\sigma,\tau}:\Met^p_2 \rightarrow L^p\Omega^0(\SEnd \cE)$ as 

\[
m_{\sigma,\tau}(H) = \sigma_v\sqrt{-1} F_H + [\phi,\phi^{*H}] + [x,x^{*H}] - [y^{*H},y] -\tau\cdot \text{id}
\]

\noindent It is clear that if a Hermitian metric $H$ is a zero of $m_{1,\tau}(H)$ then the Nakajima representation $(A_H,\phi,x,y)$ lies in $\mu_\bbR^{-1}(\tau)$. Furthermore, it will be shown in the next section that the metrics which minimize $\cM_{\tau}$ are zeroes of $m_{1,\tau}$. This provides the connection between our Lagrangian and the moment map equations. In order to prove this we will require the following definition.

For a fixed real number $B\geq 0$, define the space 
\[
\Met_B(L^p_2(S^0))=\{ H\in \Met(L^p_2(S^0)) : \norm{m_{\sigma,\tau(H)}}^p_{L^p,H} \leq B\}.
\]

\subsubsection{Lagrangian minimizers solve moment map}\label{subsection:Result}

\begin{prp}\label{prp:Simplesolves}
If $(\cE,\phi,x,y)$ is simple and $H\in \Met_B(L^p_2(S^0))$ minimizes $\cM_{\tau}$, then $$m_{\sigma,\tau}(H)=0.$$
\end{prp}

Together with \Cref{prp: stable implies simple}, \Cref{prp:Simplesolves} shows that a stable $\overline{Q}$-bundle admits a hermitian metric solving \cref{eq:1}. The proof of \Cref{prp:Simplesolves} will rely on certain bounds for selfadjoint sections $s$.

In order to complete the proof some results involving $\cM_{\tau}$ are necessary. Let $L_H:L^p_2S(H)\rightarrow L^pS(H)$ be given by

\[
L_H(s) = \frac{d}{d\epsilon} m_{\sigma,\tau}(He^{\epsilon s})\lvert_{\epsilon=0}.
\]

A straightforward computation shows that

\[
L_H(s) = \sqrt{-1}\bar{\partial}_\cE\partial_H s + [\phi,[s,\phi]^{*H}] + [x,[s,x]^{*H}] + [[s,y]^{*H},y].
\]

\begin{lem}
	\begin{enumerate}
		\item $\frac{d}{d\epsilon}\cM_{\sigma,\tau}(H^{\epsilon s})\lvert_{\epsilon=0} = (m_{\sigma,\tau}(H),s)_{L^2,H}$,
		\item $\frac{d^2}{d\epsilon^2}\cM_{\sigma,\tau}(He^{\epsilon s})\lvert_{\epsilon=0}=(L_H(s),s)_{L^2,H} =\frac{d^2}{d\epsilon^2}M_D(H) +\norm{[s,\phi]}^2_{L^2,H}+\norm{[s,x]}^2_{L^2,H} + \norm{[s,y]}^2_{L^2,H}$.
	\end{enumerate}
\end{lem}

\begin{proof}
In order to prove $(1)$, it is sufficient to show that $\frac{d}{d\epsilon}\abs{y}^2_{H_\epsilon}\lvert_{\epsilon=0}= -([y^{*H},y]s)$. Equality for the remaining terms is given by \cite[lem 3.3]{HKcorrespondence}. 

\begin{align*}
\frac{d}{d\epsilon} \abs{y}^2_{H_\epsilon}\lvert_{\epsilon=0}&= \frac{d}{d\epsilon}\text{tr}(y\circ y^{*H_\epsilon}) =\text{tr}(y\circ \frac{d}{d\epsilon}y^{*H_\epsilon}\lvert_{\epsilon=0}) \\
											 &=\text{tr}(y[s,y]^{*H}) \\
											 &= -\left([y^{*H},y]s\right).
\end{align*}

\noindent This proves $(1)$. The first equality in $(2)$ comes from $(1)$ and the definition of $L_H$. \cite[lem 3.3]{HKcorrespondence} gives the second equality.
\end{proof} 

We are now in a position to prove \Cref{prp:Simplesolves}.

\begin{proof}

We start by showing that, when restricted to $L^P_2S^0(H)$, $L_H$ is a Fredholm operator of index zero and kernel zero. That $L_H$ is a Fredholm operator of index zero is contained in the proof of \cite[Prop. 3.1]{HKcorrespondence}. To see that the kernel of $L_H$ must be $0$, assume $L_H(s)=0$. Then $(s,L_H(s))_{L^2,H}=0$ and by the previous lemma $[s,\phi]=[s,x]=[s,y]=0$ on coordinate overlaps. By the gluing conditions $[s,\phi]=[s,x]=0$ everywhere, so $s$ is an endomorphism of the simple $\overline{Q}$-bundle $(\cE,\phi,x,y)$, hence $s$ is a scalar multiple of the identity. Moreover $s$ being trace-free implies $s_v=0$ for all $v$. Therefore $L_H$ is surjective when restricted to $L^p_2S^0(H)$.

Now, suppose $H$ minimizes $\cM_{\tau}$. Assume that $m_{\sigma,\tau}(H) \neq 0$. Since $L_H$ is surjective there exists a section $s\in L^p_2S^0(H)$, non-zero such that $L_H(s)=-m_{\sigma,\tau}(H)$. Let $H_\epsilon = He^{\epsilon s}$ be a path in $\Met(L^p_2(S^0))$. For $p=2n$

\begin{align*}
\frac{d}{d\epsilon} \norm{m_{\sigma,\tau}(H)}^p_{L^2,H_\epsilon}\lvert_{\epsilon=0} &= n \int_X \frac{d}{d\epsilon} \abs{m_{\sigma,\tau}(H_\epsilon)}^2_{H_\epsilon}\abs{m_{\sigma,\tau}(H_\epsilon)}^{2n-2}_{H_\epsilon}\lvert_{\epsilon=0} \\
														&=(-2n)\int_X \abs{m_{\sigma,\tau}(H)}^{2n}_H = -p\norm{m_{\sigma,\tau}(H)}^p_{L^p,H} <0.
\end{align*}

\noindent Thus for small $\epsilon$, $H_\epsilon\in \Met_B(L^p_2(S))$ and since $H$ minizimes $\cM_{\sigma,\tau}$, 

\[
\frac{d}{d\epsilon} \cM_\tau(H_\epsilon)\lvert_{\epsilon=0} = 0.
\]

 By previous computations it follows that $(L_H(s),s)_{L^2,H} =0$. However, if $(E,\phi,x,y)$ is simple and $(L_H(s),s)=0$ then $s=0$, contradicting the assumption that $-m_{\sigma,\tau}(H)=-L_H(s)\neq 0$. 
\end{proof}

It has now been shown that if $(\cE,\phi,x,y)$ is a simple $\overline{Q}$-bundle, and $H$ is a metric on $\cE$ which minimizes the Lagrangian $\cM_{\sigma,\tau}$, then the representation $(A_H,\phi,x,y)$ satisfies \cref{eq:1}. The proof of \Cref{thm:1} will be complete once the existence of minimizing metrics is shown.

\subsubsection{Existence of minimizing metrics}\label{subsubsection:Existenceofminimizingmetrics}

\begin{prp}\label{prp:estimateimpliesmetric}
If $(\cE,\phi,x,y)$ is simple and there exist constants $C_1(B),C_2(B)>0$ such that for all $s\in L^p_2(S)$ and $H\in \Met_B(L^p_2(S^0))$, we have

\[
\sup \abs{s} \leq C_1 \cM_{\tau}(H) +C_2,
\]

\noindent then there is a projectively unique hermitian metric satisfying $m_{\sigma,\tau}=0$.
\end{prp}

\begin{proof}
The proof of \cite[\S 3.14]{BradlowSpecialMetrics} extends to the current work.
\end{proof}

\begin{prp}\label{prp:bounds}
There exist constants $C_1,C_2>0$ such that $\sup\abs{s} \leq C_1\norm{s}_{L^1} + C_2$ for all $s\in L^p_2(S)$ and $H\in \Met_B(L^p_2(S^0))$.
\end{prp}

By the previous two propositions, to prove \Cref{prp:Simplesolves} it is sufficient to prove that there exist constants $C_1,C_2$, depending on the real number $B$ such that $\norm{s}_{L^1} \leq C_1\cM_{\sigma,\tau} +C_2$. This requires a series of lemmas. The first lemma, a result of Donaldson \cite{DonaldsonTwisted}, applies to the current setting without change.

\begin{lem}\label{lem:DonaldsonTwisted} 
    There exists a smooth function $a:\bbR_{\geq 0} \rightarrow \bbR_{\geq 0}$ satisfying the following: $a(0)=0$, for $x\geq 1$, $a(x)=x$. For any $M\in \bbR$ there is a constant $C_M$ such that if $f$ is a positive bounded function on $X$ and $\Delta f\leq b$, where $b$ is a function in $L^p(X)$ for $p>n$ with $\norm{b}_{L^p} \leq M$, then $\sup \abs{f} \leq C_Ma(\norm{f}_{L^1})$. Moreover, if $\Delta f\leq 0$ then $\Delta f=0$.
\end{lem}

The following two lemmas require only minor modifications in order to account for the sections $y$ which are not defined globally.

\begin{lem}\label{lem: inequalities}
      Let $i\in \{x,\phi\}$, and $H=Ke^s\in \Met(L^p_2(S))$, for $s\in L^p_2S$. Then the following inequalities hold:

     \[
     ([i,i^{*K}],s) \leq ([i,i^{*H}],s), \quad ([y^{*K},y],s) \geq ([y,^{*H},y],s).
     \]
\end{lem}

\begin{proof} This is \cite[lem 3.5]{HKcorrespondence}, extended to double quivers. Define functions $f_i(\epsilon) = ([i,i^{*H_\epsilon}],s)$ and $f_y(\epsilon)=([y^{*H_\epsilon},y],s)$. The $f_i$ are increasing functions, given that $\frac{df_i(\epsilon)}{d\epsilon}= \abs{[s,i]}^2_{H_\epsilon}$ with $f_i(0)= ([i,i^{*K}],s)$ and $f_i(1) = ([i,i^{*H+\epsilon}],s)$, whereas $f_y$ is decreasing, as $f_y(0)=([y^{*K},y],s)$ and $f_y(1)=([y^{*H},y],s)$.
\end{proof}

\begin{lem}
    For $H=Ke^s$ as above and $\Delta$ the Laplacian,

    \[
    (m_{\sigma,\tau}(H) - m_{\sigma,\tau}(K),s) \geq \frac{1}{2}\abs{\sigma^{1/2}\cdot s}\Delta\abs{\sigma^{1/2}\cdot s}.
    \]
\end{lem}

\begin{proof}
     Let $\Lambda$ denote the adjoint of forming the wedge product with the symplectic form $\omega$. Using \cref{eq: connection} and \Cref{lem: inequalities} the proof is identical to that in \cite[lem 3.6]{HKcorrespondence}. We have

    \begin{align*}
    (m_{\sigma,\tau}(H)- m_{\sigma,\tau}(K),s) &\geq \sqrt{-1}\Lambda(\sigma F_H-\sigma F_K,s) \\
                             &= \sqrt{-1}\Lambda(\sigma(F_K + \overline{\partial}_E(e^{-s}\partial_K(e^s))-\sigma F_K,s) \\
                             &=\sqrt{-1}\Lambda(\sigma \overline{\partial}_E(e^{-s}\partial_K e^s),s).
    \end{align*}

    Here, 
    
    \begin{equation}\label{eq:14}
    (\sigma\bar{\partial}_\cE(e^{-s}\partial_K e^s),s) = \bar{\partial}(\sigma\cdot e^{-s}\partial_Ke^s,s) + (\sigma\cdot e^{-s}\partial_K e^s,\partial_K s).
    \end{equation}
    
    Choosing a local $K_v$-orthogonal basis $(u_{v,i})$, a calculation shows $(e^{-s_v}\partial_{K_v}e^{s_v},s_v) = \frac{1}{2}\partial\abs{s_v}^2$. Multiplying by $\sigma_v$, summing over all $v\in V(Q)$ and applying $\bar{\partial}$,

    \begin{equation}\label{eq:15}
        \bar{\partial}(\sigma\cdot e^{-s}\partial_Ke^{-s},s) =\frac{1}{2}\bar{\partial} \partial \abs{\sigma^{1/2}\cdot s}^2 =\abs{\sigma^{1/2}\cdot s}\bar{\partial}\partial\abs{\sigma^{1/2}\cdot s}+\bar{\partial}\abs{\sigma^{1/2}\cdot s}\wedge\partial \abs{\sigma^{1/2}\cdot s}.
    \end{equation}

The equality $\Delta= 2 \sqrt{-1}\bar{\partial}\partial$, along with \cref{eq:14,eq:15} gives the inequality

\begin{eqnarray}
(m_{\sigma,\tau}(H)-m_{\sigma,\tau}(K),s) & \geq & \frac{1}{2}\abs{\sigma^{1/2}\cdot s}\Delta \abs{\sigma^{1/2}\cdot s} + \sqrt{-1}\Lambda\left(\bar{\partial}\abs{s}\wedge \partial\abs{\sigma^{1/2}\cdot s}\right)\nonumber\\ & & + \sqrt{-1}\Lambda(\sigma\cdot e^{-s}\partial_Ke^s,\partial_K s).\nonumber
\end{eqnarray}

Applying \cite[Prop. 3.7.1]{BradlowSpecialMetrics} to the sections $s_v$, multiplying by $\sigma$ and adding over all $v\in V(Q)$,

\begin{equation}\label{eq:16}
\sqrt{-1}\Lambda(\sigma\cdot e^{-s}\partial_Ke^s,\sigma\cdot \partial_Ks) \geq \sum_{v,i}\sqrt{-1}\Lambda(\partial\sigma^{1/2}_v\lambda_{v,i} \wedge \bar{\partial}\sigma^{1/2}_v\lambda_{v,i}).
\end{equation}

An application of \cite[line 3.43]{BradlowSpecialMetrics} to $\sigma^{1/2}\cdot s$ establishes

\begin{equation}\label{eq:17}
\sum_{v,i}\sqrt{-1}\Lambda(\partial\sigma^{1/2}_v\lambda_{v,i} \wedge \bar{\partial}\sigma^{1/2}_v\lambda_{v,i}) \geq -\sqrt{-1}\Lambda(\bar{\partial}\abs{\sigma^{1/2}\cdot s} \wedge \partial\abs{\sigma^{1/2}\cdot s}).
\end{equation}

Finally, combining \cref{eq:15,eq:16,eq:17} proves the lemma.
    
\end{proof}

This now allows us to prove \Cref{prp:bounds}. The proof is analogous to that for quiver sheaves and so only a sketch is given. The interested reader can see the proof of \cite[Prop. 3.3]{HKcorrespondence}.

\begin{proof}[Proof of \Cref{prp:bounds}]
   First it is demonstrated that the hypotheses of \Cref{lem:DonaldsonTwisted} are satisfied for the function $f=\abs{\sigma^{1/2}\cdot s}$ and suitably defined $b$. Let $C$ be the constant such that $\sup f \leq Ca(\norm{f}_{L^1})$, where $a$ is a function as in \Cref{lem:DonaldsonTwisted}. Rewriting this bound, and using the fact that $\abs{s}\leq \sigma_{\text{min}}^{-1/2}f$ the result follows up to redefining constants. 
\end{proof}

At this point all that remains is to show that stable $\overline{Q}$-bundles admit hermitian metrics satisfying the necessary constraint given in \Cref{prp:estimateimpliesmetric}.

\begin{prp}\label{prp:C^0estimate}
     If $(\cE,\phi,x,y)$ is stable then there exist constants $C_1,C_2>0$ such that for all $H,s$, $\sup\abs{s} \leq C_1\cM_{\sigma,\tau}+C_2$.
\end{prp}

The proof once again uses a series of lemmas, however at this point all the details are identical to \cite[\S 3.6]{HKcorrespondence} using results carried over to double quivers. 
The idea of the proof is to assume $\cM_\tau$ does not satisfy the desired estimate. We can then construct a destabilizing subbundle of $(\cE,\phi,x,y)$ using the regularity theorem of Uhlenbeck and Yau \cite[\S 7]{Uhlenbeck-Yau}, contradicting stability of $\cE$.

 \Cref{thm:1} will then follow from \Cref{prp:C^0estimate} and \Cref{prp:Simplesolves}. Therefore, if $\cE$ is stable, then it satisfies the desired estimate, and by \Cref{prp:estimateimpliesmetric} $\cE$ admits a hermitian metric satisfying the moment map equations.

\section{Torus action}\label{section: Torus Action}

In the presence of a suitable action of a torus $\bbT$ on a (possibly singular) variety $X$ the theory of Bia\l ynicki-Birula \cite{Bialynicki-Birula73} says that the homology of $X$ is determined by the homology of the connected components of the fixed point set $X^\bbT$. In this section we define a such a torus action on Nakajima bundle varieties in an attempt to better understand the fixed points, and as a result the topology of $\cM_{\ol{Q}}^{r,d}(\tau)$.

Generalizing the $\bbC^\times$ action on Higgs bundles, there is a $\bbC^\times$ action on $\cM_{\ol{Q}}^{r,d}(\tau)$. If $(A,\phi,x,y)$ is a Nakajima representation of Q, the action of $\lambda\in \bbC^\times$ is given by$$\lambda\cdot(A,\phi,x,y) =(A,\lambda \phi,\lambda x,\lambda y).$$This action preserves both the real and complex moment maps $\mu_\bbR=\tau_\bbR,\; \mu_\bbC = 0$.
\par

If $(A,\phi,x,y)$ is a fixed point, there exists a gauge transformations $g_\lambda$ such that$$g_\lambda(A,\phi,x,y)g_\lambda^{-1} = (A, \lambda\phi, \lambda x, \lambda y).$$For each $\cE_v\in V(Q)$ let $\cE_v=\oplus_i \cE_v^{w_i}$ be the weight decomposition of $\cE_v$, in which $\bbC^\times$ acts on $\cE_v^{w_i}$ with weight $w_i$. Being a fixed point of the action implies that

\begin{align*}
    \phi_v &:\cE_v^{w_i} \longrightarrow \cE_v^{w_i-1}\otimes \Omega^{1,0}(X,\SEnd(\cE_v)), \\
    x_a    &: \cE_{t(a)}^{w_i} \longrightarrow \cE_{h(a)}^{w_i}, \\
    y_a    &: \cE_{h(a)}^{w_i} \longrightarrow \cE_{t(a)}^{w_i-1} \otimes K.
\end{align*}

If $\abs{V(Q)}=n$ one can generalize the $\bbC^\times$ action to define an action of the torus $\mathbb{T}\coloneqq (\bbC^\times)^n$ on $\cM_{\ol{Q}}^{r,d}(\tau)$. For $\lambda= (\lambda_1,\cdots, \lambda_n)\in \bbT$, where $\lambda$ acts by scaling $x_a$ by $\lambda_{h(a)}$ and scaling $\phi,y$ by $\lambda_{t(a)}$. When $Q$ consists of a single vertex this recovers the $\bbC^\times$ action above. In particular if $\ol{Q}$ contains a single vertex with no loops this recovers the ordinary $\bbC^\times$-action on the moduli space of Higgs bundles.
\par

A representation $(A,\phi,x,y)$ is now a fixed point if there exists a homomorphism $g:(\bbC^\times)^n\rightarrow \cG$ such that

\begin{equation}\label{eq:fixed}
    \lambda_k\phi_k = g(\lambda)\phi_k g(\lambda)^{-1}, \quad
    \lambda_{h(a)}x_a =g(\lambda)_{h(a)}x_a g(\lambda)_{t(a)}^{-1}, \quad
    \lambda_{t(a)}y_a = g(\lambda)_{t(a)}y_a g(\lambda)_{h(a)}^{-1}.
\end{equation}

Once again, let the weight decomposition of $\cE_v$ be given by

\[
\cE_v = \bigoplus_{(w_1,\cdots, w_n)\in \bbZ^n} \cE_v^{(w_1,\cdots, w_n)}.
\]

The fixed point condition given in \cref{eq:fixed} gives the following conditions for the sections $\phi,x,y$:

\begin{align*}
    \phi_v &: \cE_j^{(w_1,\cdots,w_n)} \longrightarrow \cE_v^{(w_1,\cdots,w_j-1,\cdots, w_n)} \otimes \Omega^{1,0}(X,\SEnd(\cE)), \\
    x_a &: \cE_v^{(w_1,\cdots,w_n)} \longrightarrow \cE_v^{(w_1,\cdots,w_{t(a)}-1,\cdots, w_n)}, \\
    y_a &: \cE_v^{(w_1,\cdots,w_n)} \longrightarrow \cE_v^{(w_1,\cdots,w_{h(a)}-1,\cdots, w_n)}, \\
\end{align*}

In particular, this shows that the Higgs field, $\phi_v$ must be nilpotent. 

This provides a partial description of the fixed points of the torus action on $\cM_{\ol{Q}}^{r,d}(\tau)$. Specifically, in order for $(\cE,\phi,x,y)$ to be fixed by the torus action each $\phi_v$ must be nilpotent. Defining the nilpotent cone inside of $\cM_{\overline{Q}}(\tau)$ to be those Nakajima representations with all Higgs fields $\phi_v$ nilpotent, all fixed points for the torus action are constrained to live inside the nilpotent cone. More investigation is required in order to produce a full description of the set of fixed points. However once this is accomplished, using the Bia\l ynicki-Birula decomposition of $\cM_{\ol{Q}}^{r,d}(\tau)$ in terms of the fixed point sets should allow for a more detailed understanding of the topology of Nakajima bundle varieties.

On a more speculative note, the fixed points of the $\bbC^\times$ action on the moduli space of Higgs bundles play an important role in understanding the Hitchin system. The Hitchin system leads to a proper morphism $\cH:\cM^{\text{Higgs}}(X)\rightarrow \cB$ to an affine space $\cB$, sending a Higgs field to the coefficients of its characteristic polynomial. The fixed points of the $\bbC^\times$ action all reside in the nilpotent cone $\cH^{-1}(0)$ of Higgs bundles $(\cE,\phi)$ with nilpotent $\phi$. Given the presence of integrable systems for Nakajima quiver varieties \cite{FisherRayan,RayanSchaposnik}, as well as the similarities between Nakajima bundle varieties and Higgs moduli, the fixed points might allow for a deeper understanding of the relationship between these systems.

\section{Examples of $\ol{Q}$-bundles}\label{section:examples}

In this section we look at examples of $\ol{Q}$-bundles, some of which are novel, and some which recover familiar objects.

Let $Q$ be the quiver with a single vertex labeled $(r,d)$ and no arrows. Let $X=\bbP^1$. A bundle representation on $X$ is a rank $r$, degree $d$, complex vector bundle $\cE$ equipped with a hermitian connection $A$, together with a section $\phi\in H^0(X,\SEnd(\cE)\otimes K)$. The moment map equations reduce to the Hitchin equations \cref{eq:Hitchin}. Thus a Nakajima bundle representation on $\bbP^1$ is a Higgs bundle on $\bbP^1$. The corresponding Nakajima bundle variety is therefore empty.

As a sanity check, we specialize to the case $r=2$. Rather than studying the solutions $(A,\phi)$, we will look at Nakajima representations from the point of view of stable $\ol{Q}$-bundles. By the Grothendieck-Birkhoff theorem \cite{Grothendieckthem}, any vector bundle on $\bbP^1$ decomposes as a direct sum of line bundles, so $\cE\cong \cO(d_1)\oplus \cO(d_2)$, and the Higgs field may be written as 

\[
\phi = \begin{bmatrix} \phi_{11} & \phi_{12} \\
				\phi_{21} & \phi_{22} \end{bmatrix},
\]

\noindent where $\phi_{ij}$ is a holomorphic section of $\SHom(\cO(d_j),\cO(d_i)\otimes K)$. Without loss of generality assume that $d_2\geq d_1$. Since any line bundle of negative degree has no nonzero sections, $\phi_{11},\phi_{22},\phi_{12}=0$. Hence $\cO(d_2)$ is a $\phi$-invariant subbundle of $\cE$ with 

\[
d_2=\mu(\cO(d_2)) \geq \mu(\cE) = \frac{d_1+d_2}{2}.
\]

\noindent Therefore there are no stable $\overline{Q}$ bundles of $\bbP^1$ of rank $2$ recovering the well known fact that there are no stable Higgs bundles of rank $2$ over $\bbP^1$.

Now consider $Q$ with a loop added to its vertex. A $\overline{Q}$-bundle is a tuple $(\cE,\phi,x,y)$ with $x\in H^0(X,\SEnd(\cE))$, $y\in H^1(X,\SEnd(\cE)\otimes K)$ and $\phi\in H^0(X, \SEnd(\cE)\otimes \Omega^{1,0})$. Let $\cE \cong \cO(d_1)\oplus \cO(d_2)$. If $\phi$ is holomorphic,
 $\phi=\begin{bmatrix} 0 & 0 \\
		\phi_{21} & 0 \end{bmatrix}$ 
  and $\cO(d_2)$ is $\phi$-invariant. The complex moment map states that $x,y$ commute. Therefore up to the $\cG$ action, $x,y$ are diagonal matrices. Then $\cE$ is not stable as $\cO(d_2)$ is an invariant subbundle with $\mu(\cO(d_2))\geq \mu(\cE)$.

The above examples illustrate that in order to find new examples of stable objects, the condition that $\phi$ be holomorphic must be weakened. Up to the action of $\cG$ we have the following descriptions of the sections $x$ and $y$,

\[
x= \begin{bmatrix} x_{11} & 0 \\
				x_{21}(z) & x_{22}\end{bmatrix}, \quad
y = \begin{bmatrix} y_{11} & y_{12} \\ 
                    y_{21}(z) & y_{22} \end{bmatrix} dz\wedge d\bar{z},
\]

where the diagonal entries are constant and $y_{12}$ is either constant or $0$ depending on whether $d_2=d_1$ or $d_2>d_1$. In the latter case, from the equation $\overline{\partial}\phi = -[x,y]$, we see that $\phi$ can have at most linear dependence on $\bar{z}$ and therefore has the form

\begin{equation}\label{eq:unstablephi}
\phi =\begin{bmatrix} 0 & 0 \\
(x_{21}y_{11}+x_{22}y_{21}-x_{11}y_{21}-x_{21}y_{22})\ol{z}+ \phi_{22} & 0
\end{bmatrix}dz,
\end{equation}

for $\phi_{22}\in H^0(X,\SHom(\cO(d_1),\cO(d_2))\otimes K)$. From this, it can be seen that $\cO(d_2)$ is a destabilizing subbundle of $\cE$. Thus no stable $\ol{Q}$-bundles exist in the case $d_2>d_1$. However, if $d_2=d_1$ then the most general form for $\phi$ is

\begin{equation}\label{eq:stablephi}
    \phi=\begin{bmatrix}
        (-x_{21}y_{12})\ol{z} & (x_{11}y_{12}-x_{22}y_{12})\ol{z} \\
        (x_{21}y_{11}+x_{22}y_{21}-x_{11}y_{21}-x_{21}y_{22})\ol{z} & (x_{21}y_{12})\ol{z}
    \end{bmatrix} dz.
\end{equation}

Taking the constants $x_{ii},y_{ii},y_{12}$ appropriately, many examples of stable $\ol{Q}$-bundles may be realized in this way.

Notice that when $y_{12}=0$ then \cref{eq:stablephi} has the same form as \cref{eq:unstablephi}, hence we see a necessary condition for $(\cE,\phi,x,y)$ to be a stable $\ol{Q}$-bundle is that $x,y$ not be simultaneously triangularizable. Up to the $\cG$ action stable $\ol{Q}$-bundles are of the form

\begin{equation}
    x=\begin{bmatrix} x_{11} & x_{12} \\
                        0    & x_{22}
      \end{bmatrix},
    y=\begin{bmatrix} y_{11} & y_{12} \\
                      y_{21} & y_{22}
      \end{bmatrix},
    \phi=\begin{bmatrix} 0 & (x_{11}-x_{22})y_{12}\ol{z} \\
                    (x_{22}-x_{11})y_{21}\ol{z} & 0
        \end{bmatrix},
\end{equation}  

where $y_{12}\neq 0$. A quick check shows that $\text{dim}_\bbC\cM_{\ol{Q}}^{r,d}(\tau)=7$.

For the next example consider the $A_2$ quiver, having two vertices connected by a single edge. A stable $Q$-bundle on a Riemann surface $X$ is a pair of vector bundles $\cE_1,\cE_2$ together with a holomorphic morphism $x:\cE_1\rightarrow \cE_2$, i.e. a holomorphic triple on $X$. These objects, as well as the moduli spaces they give rise to, have been thoroughly studied, for instance in \cite{HolomorphicTriples}.

To see what the construction of $\ol{Q}$-bundles might yield, let us assume the quiver has label $((1,d),(2,d'))$. In the case $X=\bbP^1$, a $\ol{Q}$-bundle consists of two vector bundles $\cE_1= \cO(d)$ and 
\newline $\cE_2= \cO(d_1)\oplus \cO(d_2)$, where $d_1+d_2=d'$, with the appropriate sections $\phi_1,\phi_2,x,y$ subject to the conditions

\begin{equation}\label{eq:conditions}
\overline{\partial}_{\cE_1}- yx =0, \quad \overline{\partial}_{\cE_2} +xy=0.
\end{equation}

In order for stable objects to exist assume $d_2>d$. If $d>d_2\geq d_1$, both $x,y=0$ and the $\ol{Q}-$bundle reduces to a pair of Higgs bundles on $\bbP^1$. On the other hand, if $d>d_1$, then degree restrictions together with \cref{eq:conditions} imply every $\ol{Q}$-bundle contains $(\cO(d),\cO(d_2),\phi_1,\phi_2,x,y)$ as a destabilizing subbundle. 

This leaves the case $d_2\geq d_1 \geq d$. It can be checked that stable $\ol{Q}$-bundles are given by

\begin{align}
x=\begin{bmatrix} x_1 \\ x_2 \end{bmatrix}, &\quad
y= \begin{bmatrix} y_1 & y_2 \end{bmatrix} dz,\\
\phi_1 = ( x_1y_1 + x_2y_2)\ol{z}dz, &\quad 
\phi_2 = \begin{bmatrix} x_1y_1\ol{z} & x_1y_2\ol{z} \\
                         x_2y_1\ol{z} +\phi_{22} & x_2y_2\ol{z} \end{bmatrix}.
\end{align}

Thus in order for stable bundles to exist, it must be that $d'\geq 2d$. Taking $d'=2d$, the expected complex dimension $\cM_{\ol{Q}}^{r,d}(\tau)$ is 

\[
2(d_1-d+1)+2(d_2-d+1) =4.
\]

Finally, moving away from $\bbP^1$ let $X$ be a Riemann surface of genus $g\geq 2$. Let $Q$ be the quiver with a single looped vertex. Fix $\theta$, a square root of the canonical bundle. Taking the vector bundle $\cE=\theta \otimes \theta^{-1}$ a stable $\ol{Q}$-bundle is given by the sections

\[
\phi = \begin{bmatrix} 0 & \alpha \\
                        1 & 0
        \end{bmatrix}, \quad 
x   = \begin{bmatrix} 1 & 0 \\
                      \beta & -1 
     \end{bmatrix},\quad 
y   = \begin{bmatrix}
        y_{11} & 0 \\
        \delta & -y_{11}
        \end{bmatrix}.
\]
 
Here $\alpha$ is interpreted as a quadratic differential on $X$, $\beta$ a holomorphic $1$-form, and $\delta \in H^1(X,K^2)$.
\par

Denote an infinitesimal deformation of the corresponding Nakajima representation $(A,\phi,x,y)$ by $(\dot{A},\dot{\phi},\dot{x},\dot{y})$. Such a deformation arises from an element $\alpha\in \Omega^0(X,\SEnd(\cE))$ whenever

\begin{equation}\label{eq:d_1}
    \dot{A}=d_A\alpha, \quad \dot{\phi}=[\phi,\alpha], \quad \dot{x} = [x,\alpha],\quad \dot{y}= [y,\alpha].
\end{equation}

We obtain a $2$-term complex 

\begin{equation}
\begin{tikzcd}
    \Omega^0(X,\SEnd(\cE)) \arrow[r,"d_1"] &A \arrow[r,"d_2"] & \Omega^2(X,\SEnd(\cE)) \oplus \Omega^2(X,\SEnd(\cE)\otimes \bbC)
\end{tikzcd},
\end{equation}

where 

\begin{equation}
    A= \Omega^1(X,\SEnd(\cE)) \oplus \Omega^{1,0}(X,\SEnd(\cE)\otimes \bbC) \oplus \Omega^0(X,\SEnd(\cE)) \oplus \Omega^{1,0}(X,\SEnd(\cE)\otimes \bbC).
\end{equation}

The map $d_1$ is given by \cref{eq:d_1} and $d_2$ by the linearization of the moment maps. An index-theoretic argument similar to that in \cite[Sec. 5]{HitchinSelf-Duality} demonstrates that the expected dimension of $\cM_{\ol{Q}}^{r,d}(\tau)$ is $6(g-1)=6$.

\bibliographystyle{alpha}
\bibliography{Bibliography}

\pagebreak

\Addresses

\end{document}